\documentclass[12
pt]{article}
\usepackage{pslatex}
\usepackage{fancyhdr}
\usepackage{graphicx}
\usepackage{geometry}
\RequirePackage[latin1]{inputenc}\RequirePackage[T1]{fontenc}

\def\figurename{Figure} 
\makeatletter
\renewcommand{\fnum@figure}[1]{\figurename~\thefigure.}
\makeatother

\def\tablename{Table} 
\makeatletter
\renewcommand{\fnum@table}[1]{\tablename~\thetable.}
\makeatother
\usepackage{color}
                [2000/03/29 v1.0m Standard LaTeX package]

\usepackage{amsmath}
\usepackage{amssymb}
\usepackage{amsfonts}
\usepackage{amsthm,amscd}
\newtheorem{theorem}{Theorem}[section]

\theoremstyle{definition}

\newtheorem{definition}[theorem]{Definition}
\newtheorem{example}[theorem]{Example}

\theoremstyle{remark}
\newtheorem{remark}[theorem]{Remark}

\numberwithin{equation}{section}

\def\P{\mathbb P}
\def\R{\mathbb R}
\def\E{\mathbb E}

\def\Q{\mathbb Q}

\def\E{\mathbb E}

\begin{document}

\title{Explicit solution for backward stochastic Volterra integral equations with linear time delayed generators\thanks{This work is supported by the National Natural Science Foundation of China (N° 11871076)}}

\author{Y. Ren$^{\,a,}$\thanks{yongren@126.com}\;\;  H. Coulibaly$^{\,b,}$ \thanks{harounbolz@yahoo.fr}\;\; and\;\; A. Aman$^{\,b,}$ \thanks{augusteaman5@yahoo.fr\;/\; aman.auguste@ufhb.edu.ci, corresponding author}\\
a. Anhui Normal University, Department of Mathematics, Wuhu, China\;\;\;\;\;\;\;\;\;\;\;\;\;\;\;\;\;\;\;\;\;\;\;\\
b. Université F. H. Boigny, UFR Mathématiques et Informatique, Abidjan, Côte d'Ivoire}

\date{}
\maketitle \thispagestyle{empty} \setcounter{page}{1}

\thispagestyle{fancy} \fancyhead{}
 \fancyfoot{}
\renewcommand{\headrulewidth}{0pt}

\begin{abstract}
This note aims to give an explicit solution for backward stochastic Volterra integral equations with linear time delayed generators. The process $Y$ is expressed by an integral whose kernel is explicitly given. The processes $Z$ is expressed
by Hida-Malliavin derivatives involving $Y$.
\end{abstract}

\vspace{.08in} \noindent \textbf{MSC}:Primary: 60F05, 60H15; Secondary: 60J30\\
\vspace{.08in} \noindent \textbf{Keywords}: Backward stochastic Volterra integral equation, time delayed generator, Linear equation, Explicit solution, Hida- Malliavin derivative.

\section{Introduction}
\label{intro}
\setcounter{theorem}{0}\setcounter{equation}{0}
Backward stochastic Volterra integral equations (BSVIEs, for short) have been initiated in \cite{P8} under Lipschitz condition. This assumption has been relaxed to local Lipschitz condition in \cite{P9}. A few years later, the completed theory of backward stochastic Volterra integral equations (BSVIEs, for short) has been introduced by Yong in \cite{Yong,Yong1} and references therein. This king of BSDEs has been connected to optimal control problems for controlled Volterra type systems.
Recently, in \cite{CA}, Coulibaly and Aman study under Lipschitz assumption the following BSVIEs with general time delayed generator
\begin{eqnarray}
Y(t)&=& \xi+\int_t^Tf(t,s,Y_s,Z_{t,s})ds-\int_t^TZ(t,s)dW(s).\label{BSVIE}
\end{eqnarray}
Here the process $(Y_s,Z_{t,s})=(Y(s+u),Z(s+u))_{-T\leq u\leq 0}$ denotes the past of process $(Y,Z)$ until $(t,s)$. This type of BSDEs can be viewed as the combinaison of Voltera BSDEs and delayed BSDE introduced in \cite{DI1,DI2}. This study is very interesting since BSVIEs with delayed generator are used to model a recursive utility with delay which appears when we face the problem of dynamic modeling of non monotone preferences. For more detail we refer the reader to \cite{Del} and reference therein. However, as the counter examples in \cite{CA} attest, the existence and uniqueness of the solution is not valid in general. We distinguish two conditions of existence and uniqueness. First, in the case of general probability measure, it is necessary and sufficient that, in a certain sense, the Lipschitz constant or the terminal time be small enough. Secondly, if the delay measure $\alpha$ is supported by $[-\gamma,0]$ for $\gamma>0$ sufficiently small, then existence and uniqueness result hold for all Lipschitz constant and terminal time.  

In this paper our study concern a the linear BSVIE with time delayed generator which is a special case of BSVIE \eqref{BSVIE}. Let $\{F(t) , 0\leq t\leq T \}$ be a given stochastic process, not necessarily adapted, $(G(t,s), 0\leq t\leq s\leq T)$ and $(g(s), 0\leq s\leq T)$ be given functions of $t,\,s$ with values in $\R$ and denote by $\alpha$ a probability measure on $([-T,0],\mathcal{B}([-T,0]))$. We consider the following BSVIE:
\begin{eqnarray}
Y(t)&=& F(t)+\int_t^T\int_{-T}^{0}[G(t+u,s+u)Y(s+u)+g(s+u)Z(t+u,s+u)]\alpha(du)ds\nonumber\\
&&-\int_t^TZ(t,s)dW(s).\label{BSVIElinear}
\end{eqnarray}
We remark that if $\alpha$ is the Dirac measure on $0$ i.e $\alpha=\delta_{0}$, hence \eqref{BSVIElinear} becomes
\begin{eqnarray}
Y(t)&=& F(t)+\int_t^T[G(t,s)Y(s)+g(s)Z(t,s)]ds\nonumber\\
&&-\int_t^TZ(t,s)dW(s).\label{BSVIElinear1}
\end{eqnarray}
which has been study in \cite{Hu} and next in \cite{Wal}.

The rest of this note is organized as follows. In Section 2, we introduce some fundamental knowledge and assumptions concerning the data of BSVIEs \eqref{BSVIElinear}. Section 3 is devoted to derive our result. 

\section{Preliminaries}
\setcounter{theorem}{0} \setcounter{equation}{0}
In all this paper, we shall work on a probability space $(\Omega, \mathcal{F},\P)$ equipped with the filtration ${\bf F}= (\mathcal{F}_t )_{t\geq0}$ satisfying the usual condition. Let $W = (W(t),\; t\geq 0 )$ be a standard one dimensional ${\bf F}$-adapted Brownian motion. In order to well defined a notion of solution of our BSVIEs with time delayed generator, let us consider the following spaces.      
\begin{description}
\item $\bullet $ Let $ L_{-T}^2 (\mathbb{R} ) $ denote the space of measurable functions $ z : [-T;0] \rightarrow \mathbb{R} $ satisfying
$$   \int_{-T}^0 \mid z(t) \mid^2 dt < \infty .
$$
\item $ \bullet $ Let $ L_{-T}^{\infty } (\mathbb{R} )$ denote the space of bounded, measurable functions $ y : [-T,0] \rightarrow \mathbb{R} $\\
satisfying
$$
\sup\limits_{-T\leq t\leq 0} \mid y(t) \mid < \infty.
$$

\item $\bullet$ Let $ \mathbb{H}_1:= \mathcal{H}_{[-T,T]}^2(\R)$ denote the space of all adapted process $\eta$ with values in $\R$ such that $\eta(t)=\eta(0)$ for $t<0$ and $$\|\eta\|_{\mathbb{H}_1}=\E\left[\left(\int_{-T}^T e^{\beta s}|\eta(s)|^2ds\right)\right]^{1/2}<+\infty.$$
\item $\bullet$ Let $ \mathbb{H}_2:= \mathcal{H}_{D_T}^2(\R)$ denote the space of all functions $\varphi: D_T\rightarrow \R $ such for all $t\in [-T,T]$, the process $\varphi(t,s)_{s\geq t}$ is adapted, $\varphi(t,s)=0$ if $t<0$ or $s<0$ and $$\|\varphi\|_{\mathbb{H}_2}=\E\left[\left (  \int_0^T\int_t^T e^{\beta s} |\varphi(t,s)|^2dsdt\right)\right]^{1/2}<+\infty,$$ where $D_T=\{(t,s)\in [-T,T]^2, t\leq s\}$.
\item $\bullet$ Let $\mathcal{S}^2(\R)$ denote the space of all predictable process $\eta$ with values in $\R$ such that the norm  $\displaystyle \|\eta\|_{\mathcal{S}^2}=\E\left[\sup_{0\leq s\leq T}e^{\beta s}|\eta(s)|^2\right]<+\infty$.
\end{description}
\begin{definition}
The process $(Y(.),Z(.,.))$ is called solution  to \eqref{BSVIElinear} if $(Y(.),Z(.,.))$ belongs to $\mathbb{H}_1\times\mathbb{H}_2$ and satisfies \eqref{BSVIElinear}.\end{definition}

Our result will be derived under the following assumptions.

\begin{description}
\item {(\bf A1)}  $\{F(t), 0\leq t \leq T\}$ is a given stochastic process, not necessarily adapted that belongs in $L^2 (\Omega,\mathcal{F}_T,\mathbb{P})$.
\item {(\bf A2)} $G: D_T\rightarrow \mathbb{R} $ is a measurable, uniformly bounded function such that $G(t,s)=0$ if $t<0$ or $s<0$.
\item {(\bf A3)} $g:[-T,T] \rightarrow \mathbb{R}$ is a measurable, uniformly bounded function such that $g(s)=0$ if $s<0$.
\end{description}
We end this section by this clarification : for each $(t,s) \in D_T,\;\; (Y(t),Z(t,s)) $ denotes the value of the process $(Y,Z)$ at $(t,s)$ while $(Y_t,Z_{t,s}) = (Y(t+u),Z(t+u,s+u))_{-T \leq u \leq 0} $ denotes all the past of $(Y,Z)$ until $(t,s) $. Therefore, for each $(t,s) \in D_T$ and almost all $\omega \in \Omega,\;\; Y_t(\omega) $ and $ Z_{t,s} (\omega) $ belong respectively to $L_{-T}^{\infty} (\mathbb{R} )$ and $L_{-T}^2(\mathbb{R})$.

\section{Main result}
Let us now derive our main result that is an explicit formula of the solution of BSVIE \eqref{BSVIElinear}. In this fact, we state a need notations and results. Let us set
\begin{eqnarray*}\label{A00}
  	\Phi(t,r)=\alpha([r-T,0])G(t,r)
\end{eqnarray*}
and consider the sequence $(\Phi^{(n)})_{n\geq 0}$ defined recursively as follows: $\Phi^{(1)}(t,r)=\Phi(t,r)$ and for all $n\geq 2$
\begin{eqnarray*}\label{A1}
\Phi^{(n)}(t,r)=\int_t^r \Phi^{(n-1)}(t,s)\Phi(s,r)ds.
\end{eqnarray*} 
\begin{remark}
Since $\alpha$ is a probability measure, there exist a constant $C>0$ (a uniform bound of the function $G$) such $|\Phi(t,r)|<C$. Moreover, by induction method we prove that for all $n\geq 2$, 
\begin{eqnarray*}
|\Phi^{(n)}(t,s)|\leq \frac{(CT)^{n}}{n!}.
\end{eqnarray*}
Hence, for all $t,s$
\begin{eqnarray*}
\sum_{n=1}^{+\infty}|\Phi^{(n)}(t,s)|<+\infty.
\end{eqnarray*}
\end{remark}
\begin{theorem}\label{linear}
Assume that $({\bf A1})$-$({\bf A3})$ hold. If the horizon time $T$ or bound of $G$ and $g$  are small enough, then BSVIE \eqref{BSVIElinear} has a unique solution $(Y,Z)\in \mathcal{S}^2 \times \mathbb{H}_2$. Next setting 
\begin{eqnarray}\label{A0}
\Psi(t,r)=\sum_{n=1}^{\infty}\Phi^{(n)}(t,r), 
\end{eqnarray} 
 we have
\begin{enumerate}
\item[(i)]
\begin{eqnarray*}
Y(t)= \mathbb{E}^{\mathbb{Q}} \left[ F(t)+\int_t^T\Psi (t,r)F(r)dr \lvert \mathcal{F}_t \right], 
\end{eqnarray*}
where $\mathbb{Q}$ is a probability measure equivalent to $\P$.
\item[(ii)] For all $0\leq t\leq T$, let us set
\begin{eqnarray}\label{A01}
U(t) = F(t) + \int_t^T \alpha \left( (r-T,0]\right)G(t,r)Y(r)dr-Y(t).
\end{eqnarray}
Then 
\begin{eqnarray*}
Z(t,s) = \mathbb{E}^{\mathbb{Q}} \left[ D_sU(t) -U(t) \int_s^T D_s g(r)dW^{\mathbb{Q}}(r) \lvert \mathcal{F}_s \right] ; \;\;0\leq t\leq T,
\end{eqnarray*}
where $D_s$ denotes the Hida-Malliavin derivatives state briefly in \cite{Hu}. 
\end{enumerate}
\end{theorem}

\begin{proof}
Let us first prove that BSVIE \eqref{BSVIElinear} admits a unique solution. In this fact, thanks to the work of Coulibaly and Aman (see \cite{CA}, Theorem 3.3), it suffices to show that its generator is delayed Lipschitz. Let denote 
\begin{eqnarray*}
f(t,s,Y_s,Z_{t,s}) = \int_{-T}^{0}[G(t+u,s+u)Y(s+u)+g(s+u)Z(t+u,s+u)]\alpha(du).
\end{eqnarray*}
For $\P \otimes \lambda $-a.e. $(\omega ,(t,s) ) \in \Omega \times D_T $ and any $ (y_t,z_{t,s}),(y'_t,z'_{t,s}) \in  L_{-T}^{\infty} (\mathbb{R})\times  L_{-T}^2 (\mathbb{R})$, we have
\begin{eqnarray*}
&&\left|f(t,s,y_s,z_{t,s})-f(t,s,y'_s ,z'_{t,s})\right|^2 \\
& \leq & 2\int_{-T}^0 G^2(t+u,s+u)\alpha(du)\int_{-T}^0 \left|y(s+u)- y'(s+u) \right|^2\alpha(du)\\
&& + 2\int_{-T}^0 g^2(s+u)\alpha(du)\int_{-T}^0 \left|z(t+u,s+u)- z'(t+u,s+u) \right|^2\alpha(du)\\
& \leq & K  \left( \int_{-T}^0\mid y(s+u) - y'(s+u) \mid^2  \alpha (du) + \int_{-T }^0\mid z(t+u,s+u) - z'(t+u,s+u) \mid^2 \alpha (du) \right),
\end{eqnarray*}
where $K$ is a constant depending only on the uniform bound of $G$ and $g$.

Now it remain to prove $(i)$ and $(ii)$. Let us begin with $(i)$. Integrating each term of BSVIE \eqref{BSVIElinear}, we obtain for all $w\in [0,T]$, 
\begin{eqnarray}\label{A2}
\int_w^T Y(t)dt& =& \int_w^T F(t)dt +\int_w^T \int_t^T\int_{-T}^{0}G(t+u,s+u)Y(s+u)\alpha(du)dsdt\nonumber\\
&& + \int_w^T\int_t^T\int_{-T}^{0}g(s+u)Z(t+u,s+u) \alpha(du)dsdt\nonumber\\
&& -\int_w^T \left( \int_t^TZ(t,s)dW(s)\right) dt.
\end{eqnarray}
Next, we apply respectively Fubini's theorem, change of variable and use the fact that $G(t,s)=Z(t,s)=g(s)=0$ for $t<0$ or $s<0$ to derive 
\begin{eqnarray}\label{A3}
&&\int_w^T \int_t^T\int_{-T}^{0}G(t+u,s+u)Y(s+u)\alpha(du)dsdt\nonumber\\
&=&\int_{-T}^{0}\int_{w+u}^{T+u} \int_{t+u}^{T+u}G(t,s)Y(s){\bf 1}_{\{t\leq s\}}dsdt\alpha(du)\nonumber\\
&=& \int_0^T \int_0^T \alpha\left((s-T),(t-w)\wedge 0]\right)G(t,s)Y(s){\bf 1}_{\{t\leq s\}}dsdt\nonumber\\ 
\end{eqnarray}
and
\begin{eqnarray}\label{A4}
&&\int_w^T\int_t^T\int_{-T}^{0}g(s+u)Z(t+u,s+u) \alpha(du)dsdt\nonumber\\&=&\int_{-T}^{0}\int_{w+u}^{T+u}\int_{t+u}^{T+u}g(s)Z(t,s){\bf 1}_{\{t\leq s\}}dsdt \alpha(du)\nonumber\\
&=& \int_{0}^T \int_0^T \alpha\left((s-T),(t-w)\wedge 0]\right)g(s)Z(t,s){\bf 1}_{\{t\leq s\}}dsdt.\nonumber\\ 
\end{eqnarray} 
Putting \eqref{A3} and \eqref{A4} to \eqref{A2}, we obtain for all $w\in [0,T]$, 
\begin{eqnarray}\label{A2bis}
\int_w^T Y(t)dt& =& \int_w^T F(t)dt +\int_0^T \int_t^T\alpha\left((s-T),(t-w)\wedge 0]\right)G(t,s)Y(s)dsdt\nonumber\\
&& + \int_0^T\int_t^T\alpha\left((s-T),(t-w)\wedge 0]\right)g(s)Z(t,s)dsdt\nonumber\\
&& -\int_w^T \left( \int_t^TZ(t,s)dW(s)\right) dt.
\end{eqnarray}
Let us set
\begin{eqnarray*}
H(w)=\int_w^T Y(t)dt,\;\; I(w)=\int_w^T\left( F(t)-\int_t^TZ(t,s)dW(s)\right)dt,
\end{eqnarray*}
and 
\begin{eqnarray*}
J(w)&=&\int_0^T \int_t^T\alpha\left((s-T),(t-w)\wedge 0]\right)(G(t,s)Y(s)+g(s)Z(t,s))dsdt.
\end{eqnarray*}
We have 
\begin{eqnarray}\label{D1}
H'(t)=I'(t)+J'(t).
\end{eqnarray}
It is not difficult to prove that
\begin{eqnarray}\label{D2}
H'(t)=-Y(t)
\end{eqnarray}
and 
\begin{eqnarray}\label{D3}
I'(t)=-F(t)+\int_t^TZ(t,s)dW(s).
\end{eqnarray}
It remains for us to look very closely $J'(t)$. For this, let us remark that 
\begin{eqnarray*}
J'(w)=\int_0^T \int_t^T\left(\int^0_{-T}\frac{\partial{\bf 1}_{(\theta,\phi(w)]}}{\partial w}(u)\alpha(du)\right)(G(t,s)Y(s)+g(s)Z(t,s))ds dt,
\end{eqnarray*}
where $\theta=(s-T)$ and $\phi(w)=(t-w)\wedge 0 $.

Thanks to analysis tools and extensive calculations, we obtain:
\begin{eqnarray*}
J'(t)=-\int_t^T\alpha([((T-s),0])(G(t,s)Y(s)+g(s)Z(t,s))ds.
\end{eqnarray*}
Therefore \eqref{D3} becomes
\begin{eqnarray}\label{A5}
 Y(t) &=&  F(t) +  \int_t^T \Phi(t,s)Y(s)ds\nonumber \\
&& - \int_t^T Z(t,s)\left[dW(s) - \alpha \left( (s-T,0]\right)g(t,s)ds \right], \;\; a.s.
\end{eqnarray} 
Let us set
\begin{eqnarray*}
M(T)=\exp\left( \int_0^T \alpha \left( (s-T,0] \right)g(s) dW(s) -\dfrac{1}{2} \int_0^T \alpha^2 \left( (s-T,0]  \right)g^2(s)ds \right).
\end{eqnarray*}
Thus since $\displaystyle  \int_0^T \alpha^2 \left((s-T,0]  \right)g^2(s)ds $ is finite, the quantity $\Q$ defined by 
\begin{eqnarray*}\label{Novikov}
\dfrac{d\mathbb{Q}}{d \mathbb{P}}\lvert_{\mathcal{F}_T}&=&M(T) 
\end{eqnarray*}
is a probability measure equivalent to $\P$. Moreover, in view of Girsanov theorem, the process $W^{\Q}$ defined by 
\begin{eqnarray*}
W^{\mathbb{Q}} (t) = W(t)-\int_0^t \alpha \left((s-T,0]  \right)g(s)ds, \;\; 0\leq t\leq T,
\end{eqnarray*}
is a $\mathbb{Q}$-Brownian motion. Therefore, according to \eqref{A5}, the process $(Y,Z)$ solution of \eqref{BSVIElinear} satisfies
\begin{eqnarray}\label{B0}
 Y(t) &=&  F(t) +  \int_t^T \Phi(t,s)Y(s)ds\nonumber \\
&& - \int_t^T Z(t,s)dW^{\mathbb{Q}^t}(s).
\end{eqnarray} 
Finally taking the conditional $\Q$-expectation with respect $\mathcal{F}_t$, we derive
\begin{eqnarray}\label{B1}
Y_t&=&\E^{\Q}\left(F(t)+ \int_t^T \Phi(t,s)Y(s)ds|\mathcal{F}_t\right)\nonumber\\
&=& \E^{\Q}\left(F(t)|\mathcal{F}_t\right)+\int_t^T \Phi(t,s)\E^{\Q^t}\left(Y(s)|\mathcal{F}_t\right)ds.
\end{eqnarray}
Now, if we set 
\begin{eqnarray*}
\overline{F}(t,r)=\E^{\Q}(F(t)|\mathcal{F}_r)\;\; \mbox{and}\;\; \overline{Y}(r)=\E^{\Q}(Y(s)|\mathcal{F}_r),\;\; r\leq s\leq T,
\end{eqnarray*}
and take a conditional $\Q$-expectation with respect $\mathcal{F}_r$ to the both side of \eqref{B1}, we obtain
\begin{eqnarray}\label{B2}
\overline{Y}(t)&=& \overline{F}(t,r)+ \int_t^T \Phi(t,s)\overline{Y}(s)ds,\;\; r\leq t\leq T.
\end{eqnarray}
On the other hand, since 
\begin{eqnarray*}
\overline{Y}(s)&=& \overline{F}(s,r)+ \int_s^T \Phi(s,u)\overline{Y}(v)dv,
\end{eqnarray*}
it follows from \eqref{B2} and \eqref{A1} that
\begin{eqnarray*}
\overline{Y}(t)&=& \overline{F}(t,r)+ \int_t^T \Phi(t,s)\overline{F}(s,r)ds\\
&&+\int^T_t\int_s^T\Phi(t,s)\Phi(s,v)\overline{Y}(v)dvds\\
&=&\overline{F}(t,r)+ \int_t^T \Phi(t,s)F(s,r)ds\\
&&+ \int^T_t\left[\int_t^v \Phi(t,s)\Phi(s,v)ds\right]\overline{Y}(v)dv\\
&=&\overline{F}(t,r)+ \int_t^T \Phi(t,s)\overline{F}(s,r)ds+\int^T_t\Phi^{(2)}(t,s)\overline{Y}(s)ds.
\end{eqnarray*}
By repeating infinitely and using the same arguments, we have 
\begin{eqnarray*}
\overline{Y}(t)&=& \overline{F}(t,r)+\sum_{n=0}^{+\infty}\int_t^T\Phi^{(n)}(t,s)\overline{F}(s,r)ds\\
&=& \overline{F}(t,r)+ \int_t^T\Psi(t,s)\overline{F}(s,r)ds,
\end{eqnarray*}
where $\Psi$ is defined by \eqref{A0}. Finally replacing $r$ by $t$, we get $(i)$.

Let us end this proof by provide $(ii)$. According to it definition and in view of \eqref{B0}, we have 
\begin{eqnarray*}
U(t)=\int^T_tZ(t,s)dW^{\Q}(s).
\end{eqnarray*}
Finally with the identical argument used by Hu and Oksendal \cite{Hu} we get $(ii)$. 
\end{proof}
Let us illustre our result by an adapted version of example appear in \cite{Hu}.
\begin{example}
As in \cite{Hu}, we consider $\rho$ a bounded function defined on $[0,+\infty[$ such that $$\mathcal{L}\rho(x)=\int_0^{+\infty}e^{-xy}\rho(y)dy$$ is the Laplace transform of $\rho$ and $\alpha$ the uniform probability measure. Let set 
\begin{eqnarray*}
G(t,s)=\frac{T(s-t)}{T-s}\rho(s-t), \;\; 0\leq t\leq s\leq T.
\end{eqnarray*}
Hence, we have $\Phi(t,s)=\psi(s-t)=(s-t)\rho(s-t)$. Thus, $\Phi^{(n)}(t,s)=\psi_n(t-s)$, where $\psi_n$ denotes the $n$ fold convolution of $\psi$. Since we have 
\begin{eqnarray*}
\mathcal{L}\psi_n(x)=\left(\mathcal{L}\psi(x)\right)^n	
\end{eqnarray*}
and 
\begin{eqnarray*}
\mathcal{L}\left(\sum^{+\infty}_{n=1}\psi_n\right)(x)=\sum^{+\infty}_{n=1}\mathcal{L}\psi_n(x),
\end{eqnarray*}
we derive
\begin{eqnarray*}
\mathcal{L}\Psi(t,s)&=&\sum_{n=1}^{+\infty}\left(\mathcal{L}\psi(s-t)\right)^n\nonumber\\
&=&\frac{\mathcal{L}\psi(s-t)}{1-\mathcal{L}\psi(s-t)}.	
\end{eqnarray*}
where in last equality, we have assumed that $|\mathcal{L}\psi(x)|<1$, for all $x>0$. Specially taking $\rho(y)=e^{-y},\;\; y>0$, we provide that 
\begin{eqnarray}\label{L}
\mathcal{L}\psi(x)=\frac{1}{(x+1)^2}.
\end{eqnarray}
Set $\Psi(t,s)=\overline{\psi}(s-t)$, it follows from \eqref{L} that
\begin{eqnarray*}
\mathcal{L}\overline{\psi}(x)&=&\sum_{n=1}^{+\infty}\left(\mathcal{L}\psi(x)\right)^n\nonumber\\
&=&\frac{\mathcal{L}\psi(x)}{1-\mathcal{L}\psi(x)}\\
&=& \frac{1}{2}\left(\frac{1}{x}-\frac{1}{x+1}\right).	
\end{eqnarray*}
Finally, by converse Laplace transform, we derive
\begin{eqnarray*}
\Psi(t,s)=\frac{1}{2}(1-e^{-(s-	t)}).
\end{eqnarray*}
and 
\begin{eqnarray*}
Y(t)=\E^{\Q}\left(F(t)+\frac{1}{2}\int_t^T(1-e^{-(s-t)})F(s)ds|\mathcal{F}_t\right).
\end{eqnarray*}  
\end{example}

The smoothness properties of $Z(t,s)$ with respect to $t$ are important in the study of optimal control (see, \cite{AO}) and the numerical solutions (see, \cite{HN} and references therein). Using the explicit form of the solution (Theorem \ref{linear}) we can give sufficient conditions for such smoothness in the linear case. 
\begin{theorem}
Suppose $g$ be deterministic, the function $F$ and $G$ are $\mathcal{C}^1$ with respect to variable $t$ and satisfy
\begin{eqnarray*}
\E^{\Q}\left[\int^T_0\left\{\int^T_t\left(F^2(t)+\alpha^2([s-T,0])G^2(t,s)+(F'(t))^2+\alpha^2([s-T,0])\left(\frac{\partial G(t,s)}{\partial t}\right)^2\right)ds\right\}dt\right]<+\infty.
\end{eqnarray*}
Then for $t<s\leq T$,
\begin{eqnarray}\label{C1}
Z(t,s)=\E^{\Q}\left[D_sF(t)+\int^T_s \alpha([r-T,0])G(t,r)D_sY(r)dr|\mathcal{F}_s\right].
\end{eqnarray}
Moreover, we get
\begin{eqnarray}\label{C2}
\E^{\Q}\left[\int^T_0\left\{\int^T_t\left(\frac{\partial Z(t,s)}{\partial t}\right)^2ds\right\}dt\right]<+\infty.
\end{eqnarray}
\end{theorem}
\begin{proof}
Since the function $g$ is deterministic, $D_sg(r)=0$ for all $s\leq r\leq T$. Thus it follows from $(ii)$ of Theorem 3.2 that
\begin{eqnarray*}
Z(t,s)=\E(D_sU(t)|\mathcal{F}_s).
\end{eqnarray*}
Next, as $D_sY(t)=0$, for $s>t$ and according to \eqref{A01} we have
\begin{eqnarray*}
\E^{\Q}(D_sU(t)|\mathcal{F}_s))=\E^{\Q^t}\left(D_sF(t)+\int_s^T\alpha([r-T,0])G(t,r)D_sY(r)dr|\mathcal{F}_s\right),
\end{eqnarray*}
which provides \eqref{C1}. Hence \eqref{C2} holds.
\end{proof}
\begin{remark}
In view of all the above, we can say that the BSVIEs with a linear delayed generator can be seen as a generalization of the BSVIEs with a linear generator. Indeed, in view of the proof of Theorem \ref{linear}, we have transformed BSVIEs with linear delayed generator \eqref{BSVIElinear} to classical linear BSVIE
with generator depending to a probability measure appear in the delayed linear generator.
\end{remark}

\section*{Acknowledgements}
This research was carried out with support of the National Natural Science Foundation of China (N°11871076).

\label{lastpage-01}
\end{document}